\documentclass[notitlepage,11pt,reqno]{amsart}
\usepackage{amssymb,amsmath,amscd,xspace,amsthm,color}
\usepackage[mathscr]{eucal}
\usepackage[breaklinks=true]{hyperref}
\usepackage{verbatim}

\def\sgn {\mathop{\rm sgn}}
\def\diag {\mathop{\rm diag}}

\def\R{\mathbb{R}}

\def\Z{\mathbb{Z}}
\def\N{\mathbb{N}}
\def\C{\mathbb{C}}

\def\calp{\mathcal{P}}

\def\ree{\mathop{\rm Re}}
\def\tr{\mathop{\rm tr}}
\newcommand{\bp}{\ensuremath{\mathbb P}}
\newcommand{\bpm}{\ensuremath{\mathbb P}^{[m]}}
\newcommand{\I}{\ensuremath{\widetilde{I}}}

\usepackage[letterpaper]{geometry}
\theoremstyle{plain}

\newtheorem{utheorem}{\textrm{\textbf{Theorem}}}

\theoremstyle{plain}

\newtheorem{theorem}{Theorem}[section]
\newtheorem{corollary}[theorem]{Corollary}
\newtheorem{proposition}[theorem]{Proposition}
\newtheorem{lemma}[theorem]{Lemma}

\theoremstyle{definition}

\newtheorem{definition}[theorem]{Definition}

\newtheorem{remark}[theorem]{Remark}

\theoremstyle{remark}

\numberwithin{equation}{section}

\begin{document}
\title{On fractional Hadamard powers of positive block matrices}
\author[Dominique Guillot \and Apoorva Khare \and Bala
Rajaratnam]{Dominique Guillot \and Apoorva Khare \and Bala Rajaratnam \\
Stanford University}
\address{Department of Mathematics, Stanford University, Stanford, CA -
94305}
\date{\today}

\begin{abstract}
Entrywise powers of matrices have been well-studied in the literature,
and have recently received renewed attention due to their application in
the regularization of high-dimensional correlation matrices. In this
paper, we study powers of positive semidefinite block matrices
$(H_{st})_{s,t=1}^n$ where each block $H_{st}$ is a complex $m \times m$
matrix. We first characterize the powers $\alpha \in \mathbb{R}$ such
that the blockwise power map $(H_{st}) \mapsto (H_{st}^\alpha)$ preserves
Loewner positivity.
The characterization is obtained by exploiting connections with the
theory of matrix monotone functions which was developed by C.~Loewner.
Second, we revisit previous work by D.~Choudhury
[Proc.~Amer.~Math.~Soc.~108] who had provided a lower bound on $\alpha$
for preserving positivity when the blocks $H_{st}$ pairwise commute. We
completely settle this problem by characterizing the full set of powers
preserving positivity in this setting.
Our characterizations generalize previous results by FitzGerald-Horn, Bhatia-Elsner, and Hiai from scalars to arbitrary block size, and in particular, generalize the Schur Product Theorem. Finally, a natural and unifying framework for studying the cases where
the blocks $H_{st}$ are diagonalizable consists of replacing real
powers by general characters of the complex plane. We thus classify such
characters, and generalize our results to this more general setting. In
the course of our work, given $\beta \in \mathbb{Z}$, we provide lower and upper
bounds for the threshold power $\alpha > 0$ above which the complex
characters $z = re^{i\theta} \mapsto r^\alpha e^{i \beta \theta}$
preserve positivity when applied entrywise to Hermitian positive
semidefinite matrices. In particular, we completely resolve the $n=3$
case of a question raised in 2001 by Xingzhi Zhan. As an application of
our results, we also extend previous work by de Pillis [Duke Math.~J.~36]
by classifying the characters $K$ of the complex plane for which the map
$(H_{st})_{s,t=1}^n \mapsto (K({\rm tr}(H_{st})))_{s,t=1}^n$ preserves Loewner
positivity.
\end{abstract}
\maketitle

\section{Introduction}

The study of positive definite matrices and of functions that preserve
them arises naturally in many branches of mathematics and other
disciplines. Given a function $f: \R \to \R$ and a matrix $A = (a_{st})$,
the matrix $f[A] := (f(a_{st}))$ is obtained by applying $f$ to the
entries of $A$. Such mappings are called \emph{entrywise} or
\emph{Hadamard functions} (see \cite[\S 6.3]{Horn_and_Johnson_Topics}).
Entrywise functions preserving Loewner positivity have been widely
studied in the literature (see e.g.~ Schoenberg \cite{Schoenberg42},
Rudin \cite{Rudin59}, Herz \cite{Herz63}, Horn \cite{Horn}, Christensen
and Ressel \cite{Christensen_et_al78}, Vasudeva \cite{vasudeva79},
FitzGerald, Micchelli, and Pinkus \cite{fitzgerald}, Hiai
\cite{Hiai2009}). The subject has recently received renewed attention due
to its importance in the regularization of high-dimensional
covariance/correlation matrices
\cite{Guillot_Rajaratnam2012, Guillot_Rajaratnam2012b, hero_rajaratnam,
Hero_Rajaratnam2012, Li_Horvath, Zhang_Horvath}.
An important family of functions is the set of power functions $f(x) =
x^\alpha$ for $\alpha > 0$. Characterizing the entrywise powers that
preserve positivity is a classical problem that has been well-studied in
the literature and is now completely resolved (see \cite{FitzHorn,
Bhatia-Elsner, Hiai2009, GKR-crit-2sided}).
A natural generalization of this problem consists of studying powers
preserving positivity when applied to block matrices (see
e.g.~\cite{Dipa_proc,Gunther,Lin20141}). More precisely, let $H :=
(H_{st})_{s,t=1}^n$ be an $mn \times mn$ Hermitian positive semidefinite
matrix, where each block $H_{st}$ is an $m \times m$ Hermitian positive
semidefinite matrix. Our first main result in this paper is a complete
characterization of the powers $\alpha$ such that the matrix
$(H_{st}^\alpha)_{s,t=1}^n$ is always positive semidefinite. Here, the
power $H_{st}^\alpha$ is computed using the spectral decomposition of
$H_{st}$.
Note that when each block of $H$ is $1 \times 1$, the problem reduces to
the classical problem of characterizing entrywise powers preserving
positivity. In contrast, when $H$ consists of only one block, every power
trivially preserves positive semidefiniteness. Surprisingly, we
demonstrate that except in trivial cases, powers do not preserve
positivity when the block size is $2$ or more. This sharply contrasts the
classical case where all powers preserve positivity beyond a certain {\it
critical exponent} (see e.g.~\cite{FitzHorn,Walch_survey}).

In a previous paper, Choudhury \cite{Dipa_proc} has studied powers
$\alpha > 0$ such that the map $(H_{st}) \mapsto (H_{st}^\alpha)$
preserves Loewner positivity, under the additional assumption that the
blocks $H_{st}$ pairwise commute. She demonstrates that every power
$\alpha \in \N \cup [mn-2, \infty)$ preserves Loewner positivity.
However, it is not clear if the bound $mn-2$ is sharp, nor which smaller
non-integer powers preserve positivity. In our second main result, we
completely answer these questions by showing that the set of powers
preserving positivity when the blocks commute is exactly $\N \cup [n-2)$.
In contrast to previous results, the answer turns out to be independent
of the block size $m$. Our result therefore shows that positivity is
actually retained at a much lower threshold (critical exponent) than was
previously thought.
We then extend this characterization to commuting Hermitian blocks that
are not necessarily positive semidefinite, by considering the odd and
even extensions of the power functions. Our characterization extends
previous work by FitzGerald and Horn \cite{FitzHorn}, Bhatia and Elsner
\cite{Bhatia-Elsner}, Hiai \cite{Hiai2009}, and Guillot, Khare, and
Rajaratnam \cite{GKR-crit-2sided}.

When studying powers of block matrices, one has to assume the blocks
$H_{st}$ are positive semidefinite for the powers $H_{st}^\alpha$ to be
well-defined. When the blocks are only Hermitian, it is natural to
replace the power functions by their odd or even extensions to $\R$ (see
Hiai \cite{Hiai2009}). Note that these functions are precisely the
Lebesgue measurable multiplicative functions on $\R$ (see
e.g.~\cite{GKR-measurable}). More generally, when the blocks $H_{st}$ are
only diagonalizable, it is natural to replace the power functions by
general Lebesgue measurable multiplicative functions on $\C$. Considering
such multiplicative functions provides a general and systematic framework
in which to study powers preserving Loewner positivity, either in the
block case, the commuting block case, or the traditional scalar setting
studied by FitzGerald and Horn, Bhatia and Elsner, and Hiai. Thus, in
Section \ref{Sprelim}, we classify all measurable multiplicative
functions on $\C$ that preserve $[0,\infty)$, and identify a natural
two-parameter family of functions $\{ \Psi_{\alpha, \beta} : \alpha \in
\R, \beta \in \Z \}$ that is used throughout the paper to generalize the
power functions.
Next, in Section \ref{Sblock} we characterize which of these functions
preserve Loewner positivity when applied blockwise to Hermitian positive
semidefinite matrices $(H_{st})_{s,t=1}^n$. In Section \ref{Scommute}, we
consider the case where the blocks $H_{st}$ pairwise commute, and
complete the characterization initiated by D.~Choudhury in
\cite{Dipa_proc}. We also demonstrate how our work can be used to
generalize previous work by de Pillis \cite{depillis_69}, by
characterizing the functions $\Psi_{\alpha,\beta}$ for which the map
$(H_{st})_{s,t=1}^n \mapsto (\Psi_{\alpha,
\beta}(\tr(H_{st})))_{s,t=1}^n$ preserves Loewner positivity.

Finally, in Section \ref{Sentrywise}, we consider the traditional setting
where each block is $1 \times 1$. For all integers $\beta \in \Z$ and $n
\in \N$, we provide lower and upper bounds for the threshold power
$\alpha > 0$ above which $\Psi_{\alpha, \beta}[-]$ preserves Loewner
positivity on $n \times n$ Hermitian positive semidefinite matrices. In
particular, when $\beta = 1$, we completely resolve the $n=3$ case of a
question raised in 2001 by Xingzhi Zhan \cite[Acknowledgment
Section]{Hiai2009}, concerning the powers $\alpha > 0$ for which
$\Psi_{\alpha, 1}[-]$ preserves Loewner positivity. Moreover, we study
the same problem for arbitrary $\beta$, which had not been previously
done in the literature. \medskip

\noindent\textbf{Notation:}
Given a subset $S \subset \C$, denote by $\bp_n(S)$ the set of $n \times
n$ Hermitian positive semidefinite matrices with entries in $S$. We
denote the complex disc centered at $a \in \C$ and of radius $R > 0$ by
$D(a,R)$. We write $A \geq 0$ to denote that $A \in \bp_n(\C)$, and write
$A \geq B$ when $A - B \in \bp_n(\C)$. We denote by $I_n$ the $n \times
n$ identity matrix, and by ${\bf 0}_{n \times n}$ and ${\bf 1}_{n \times
n}$ the $n \times n$ matrices with every entry equal to $0$ and $1$
respectively. Finally, we denote the conjugate transpose of a vector or
matrix $A$ by $A^*$. 

\section{Literature review}\label{Slit}

Entrywise powers and their properties have been studied by many authors
including Horn and FitzGerald \cite{FitzHorn}, Bhatia and Elsner
\cite{Bhatia-Elsner}, Hiai \cite{Hiai2009}, and Guillot, Khare, and
Rajaratnam \cite{GKR-crit-2sided}. Most of the known results concern
matrices with blocks of dimension $1 \times 1$. We now review two of the
most important results in the area. 

\begin{theorem}[FitzGerald and Horn, {\cite[Theorem
2.2]{FitzHorn}}]\label{TFitzHorn}
Suppose $A = (a_{st}) \in \bp_n((0,\infty))$ for some $n \geq 2$. Then
$A^{\circ\alpha} := (a_{st}^\alpha) \in \bp_n$ for all $\alpha \in \N
\cup [n-2,\infty)$. If $\alpha \in (0,n-2)$ is not an integer, then there
exists $A \in \bp_n((0,\infty))$ such that $A^{\circ\alpha} \notin
\bp_n$. More precisely, Loewner positivity is not preserved for $A = ((1
+ \epsilon st))_{s,t=1}^n$, for all sufficiently small $\epsilon =
\epsilon(\alpha,n) > 0$ for $\alpha \in (0,n-2) \setminus \N$.
\end{theorem}

Note that in Theorem \ref{TFitzHorn}, the entries of the matrix $A$ are
assumed to be positive for the power $x^\alpha$ to be well-defined. In
practice, one also commonly encounters matrices with negative and complex
entries. In order to work with matrices with real entries, the papers
\cite{Bhatia-Elsner, Hiai2009} considered the odd and even extensions of
the power functions to the real line.

\begin{definition}
Let $\alpha \in \R$. We define the even and odd extensions to $\R$ of the
power function $x \mapsto x^\alpha$ via:
\begin{equation}
\phi_\alpha(x) := |x|^\alpha, \qquad \psi_\alpha(x) := \sgn(x)
|x|^\alpha, \qquad \forall x \neq 0, 
\end{equation}
\noindent and $\phi_\alpha(0) =\psi_\alpha(0) := 0$. Also define
$f_\alpha(x) := x^\alpha$ for $x>0$, and $f_\alpha(0) := 0$.
\end{definition}

Note that the definitions of $\phi_\alpha, \psi_\alpha$ given above are
natural, as they yield the unique even and odd multiplicative extensions
to $\R$ of the standard power functions. The following result completely
characterizes the powers $\alpha$ such that $\phi_\alpha$ or
$\psi_\alpha$ preserves Loewner positivity when applied entrywise. The
reader is referred to \cite{GKR-crit-2sided} for a proof and history of
this result. 

\begin{theorem}[{Bhatia and Elsner \cite{Bhatia-Elsner}, Hiai
\cite{Hiai2009}, Guillot, Khare, and Rajaratnam
\cite{GKR-crit-2sided}}]\label{Tcrit}
Let $\alpha \in \R$ and let $n \geq 2$. Then 
\begin{enumerate}
\item $\phi_\alpha[A] \in \bp_n(\R)$ for all $A \in \bp_n(\R)$ if and
only if $\alpha \in 2\N \cup [n-2,\infty)$.
\item $\psi_\alpha[A] \in \bp_n(\R)$ for all $A \in \bp_n(\R)$ if and
only if $\alpha \in (-1+2\N) \cup [n-2,\infty)$.
\end{enumerate}

\noindent Moreover, if $f = \phi_\alpha$ or $f = \psi_\alpha$ does not
preserve positivity on $\bp_n(\R)$ for some $\alpha \in \R$, there exists
a rank $2$ matrix $A \in \bp_n(\R)$ such that $f[A] \not\in \bp_n(\R)$. 
\end{theorem}

Blockwise powers yield a generalization of the entrywise powers analysis
studied above. We now recall a sufficient condition for preserving
positivity that was shown in \cite{Dipa_proc} in the case where $H =
(H_{st})$ is a block matrix with commuting blocks $H_{st}$.

\begin{theorem}[{Choudhury, \cite[Theorem 5]{Dipa_proc}}]\label{Tdipa}
Let $H = (H_{st})$ be a given positive semidefinite $mn \times mn$
matrix,  where $\{H_{st}: 1 \leq s,t \leq n\}$ are a commuting family of
normal $m \times m$ matrices. If $H$ is positive semidefinite, then so is
$(H_{st}^\alpha)$ for all $\alpha \in \N$. If in addition each $H_{st}$
is positive semidefinite, then $(H_{st}^\alpha)$ is positive semidefinite
for all real $\alpha \geq mn-2$. 
\end{theorem}

In Section \ref{Sblock} we completely characterize the powers $\alpha$
that preserve positivity when the blocks do not necessarily commute. We
then show in Section \ref{Scommute} that the bound $\alpha \geq mn-2$ in
Theorem \ref{Tdipa} is not sharp and that the optimal bound is $\alpha
\geq n-2$. Moreover, we will demonstrate how Theorem \ref{Tdipa} can be
naturally extended to blocks $H_{st}$ that are diagonalizable.

\section{Preliminaries and main results}\label{Sprelim}

Before we proceed to characterize functions preserving Loewner positivity
for block matrices, we provide a framework in which to work with powers
of complex matrices. In order to do so, first note that the functions
$\phi_\alpha$ and $\psi_\alpha$ defined in Section \ref{Slit} are in fact
the unique non-constant Lebesgue measurable multiplicative functions on
$\R$ (see e.g.~\cite{GKR-measurable}). Since we work with complex
matrices in the present paper, it is natural to first classify the
multiplicative maps on the complex plane under mild measurability
assumptions. Such a classification has been achieved in related work
\cite{GKR-measurable}.

\subsection{Multiplicative maps on the complex plane}\label{SSmult}

Given $\alpha, \beta \in \R$, define $\Psi_{\alpha,\beta} : \C\to \C$ by: 
\begin{equation}
\Psi_{\alpha,\beta}(r \exp(i \theta)) := r^\alpha \exp(i \beta \theta) \
\forall r>0, \theta \in (-\pi,\pi], \qquad \Psi_{\alpha,\beta}(0) := 0.
\end{equation}

\noindent When $\beta \in \Z$, the maps $\Psi_{\alpha,\beta}$ are
multiplicative on $\C$ and continuous on the unit circle $S^1 := \{z \in
\C: |z| = 1\}$. Moreover, $(\alpha,\beta) \mapsto \Psi_{\alpha,\beta}$ is
a monoid homomorphism from the additive group $(\R \times \Z,+)$ to the
monoid of multiplicative maps on $\C$ (under pointwise multiplication).
The following lemma shows that the functions $\Psi_{\alpha, \beta}$ for
$\alpha \in \R$ and $\beta \in \Z$ are in fact the only non-constant
multiplicative functions from $\C$ to $\C$ that 1) are continuous on
$S^1$, 2) map the positive real axis into itself (needed to preserve
Loewner positivity), and 3) satisfy natural measurability conditions. 

\begin{lemma}\label{LCmult}
Given $R \in (1,\infty]$ and $K : D(0,R) \to \C$, the following are
equivalent.
\begin{enumerate}
\item $K$ is multiplicative on $D(0,R)$, continuous on $S^1 \subset
D(0,R)$, sends $\I := (0,R)$ to $\R$, and is Lebesgue measurable on some
subinterval $I \subset \I$ which contains $1$.

\item Either $K \equiv 0$ or $K \equiv 1$ on $D(0,R)$, or there exist
$\alpha \in \R$ and $\beta \in \Z$ such that $K \equiv
\Psi_{\alpha,\beta}$.
\end{enumerate}

\noindent Moreover, the maps $\{ \Psi_{\alpha,\beta} : \alpha \in \R,
\beta \in \Z \} \cup \{ K \equiv 1 \}$ are linearly independent as
functions on $D(0,r)$ for any $0 < r \leq \infty$.
\end{lemma}

\begin{proof}[Proof of Lemma \ref{LCmult}]
Note that $K : S^1 \to \C$ is multiplicative and continuous, hence a
character. Therefore $K : D(0,R) \to \C$ is multiplicative and
conjugation-equivariant. The result now follows from \cite[Theorem
8]{GKR-measurable}.
\end{proof}

\subsection{Main results}

Before stating the main results of the paper, we introduce some notation.
Let $S \subset \C$ and $f: S \to \C$. Given a complex diagonalizable
matrix $A$ with eigen-decomposition $A = P^{-1} D P$ and spectrum
contained in $S$, we denote by $f(A)$ the matrix $f(A) = P^{-1} f(D) P$
where $f(D)$ denotes the diagonal matrix with diagonal $f(d_{11}), \dots,
f(d_{nn})$. We denote by $\bpm_{mn}(S)$ the subset of block matrices $H =
(H_{st})_{s,t=1}^n \in \bp_{mn}(\C)$ where each block $H_{st}$ is an $m
\times m$ diagonalizable matrix with spectrum contained in $S$. Note that
when $m=1$, the set $\bpm_{mn}(S)$ reduces to $\bp_n(S)$. Given $H =
(H_{st})_{s,t=1}^n \in \bpm_{mn}(S)$, we define
\begin{equation}
f^{[m]}[H] := (f(H_{st}))_{s,t=1}^n.
\end{equation}

\noindent When $m=1$, $f^{[m]}[A]$ reduces to $f[A]$. Using this
notation, we can now state the main results of the paper.

Recall that by Theorem \ref{TFitzHorn}, a power function $x^\alpha$
preserves positivity when applied entrywise to all $n \times n$ symmetric
positive semidefinite matrices with positive entries, if and only if
$\alpha \geq n-2$ or $\alpha \in \N$. Our first main result shows that,
surprisingly, the situation is radically different when the blocks have
size greater than $1$. 

\begin{utheorem}\label{Tmain1}
Let $\beta \in \Z$ and let $m,n \geq 2$.
\begin{enumerate}
\item Given $\alpha > 0$, the matrix $f_\alpha^{[m]}[(H_{st})] =
(H_{st}^\alpha) \in \bp_{mn}(\C)$ for all $(H_{st}) \in
\bpm_{mn}([0,\infty))$, if and only if $\alpha = 1$.
If $\alpha \leq 0$, then $f_0^{[m]}[-]$ preserves positivity on
$\bpm_{mn}((0,\infty))$ if and only if $\alpha = 0$.

\item The functions $\phi_\alpha^{[m]}[-]$ do not preserve positivity on
$\bpm_{mn}(\R)$ for any $\alpha \in \R$.

\item For $\alpha \in \R$, the functions $\psi_\alpha^{[m]}[-]$ preserve
positivity on $\bpm_{mn}(\R)$ if and only if $\alpha = 1$.

\item For $\alpha \in \R$, the functions $\Psi_{\alpha, \beta}^{[m]}[-]$
preserve positivity on $\bpm_{mn}(\C)$ if and only if $\alpha = 1$ and
$\beta = \pm1$ -- i.e., $\Psi_{\alpha,\beta}(z) \equiv z$ or
$\overline{z}$.
\end{enumerate}
\end{utheorem}

A natural relaxation of the hypothesis in Theorem \ref{Tmain1} is to
assume that the blocks $H_{st}$ all commute with each other. Powers
preserving positivity when applied to block matrices where the blocks
commute have been studied by D.~Choudhury in \cite{Dipa_proc}. It is
natural to ask if the lower bound $\alpha \geq mn-2$ in Theorem
\ref{Tdipa} is sharp, or if other powers preserve positivity. We
completely settle this question in our second main result, Theorem
\ref{Tmain2}, by showing that the {\it critical exponent} is in fact
$\alpha = n-2$ and that smaller non-integer powers do not preserve
Loewner positivity. In Section \ref{Scommute} we also consider the
analogue of Theorem \ref{Tmain2} where the blocks are complex
diagonalizable. 

\begin{utheorem}\label{Tmain2}
Let $\alpha > 0$ and $m,n \geq 2$. Then $(H_{st}^\alpha) \in
\bp_{mn}(\C)$ for all $(H_{st})_{s,t=1}^n \in \bp_{mn}(\C)$ such that
$H_{st} \in \bp_m(\C)$ and the blocks $H_{st}$ commute, if $\alpha \in \N
\cup [n-2,\infty)$.
If $\alpha \not\in \N \cup [n-2,\infty)$, there exist matrices $H_{st}
\in \bp_m(\C)$ such that $(H_{st}) \in \bp_{mn}(\C)$, the blocks $H_{st}$
commute, but $(H_{st}^\alpha)$ is not positive semidefinite.
Moreover, if $\alpha < 0$, there exist real symmetric positive definite
matrices $H_{st}$, $s,t = 1, \dots, n$ such that $(H_{st})_{s,t=1}^n \in
\bp_{mn}(\R)$, but $(H_{st}^\alpha)$ is not positive semidefinite. 
\end{utheorem}

In our third main result, we consider an interesting question raised by
X.~Zhan in 2001 (see \cite[Acknowledgments]{Hiai2009}). Zhan asked if
Theorem \ref{TFitzHorn} can be generalized to matrices with complex
entries when the power functions $x^\alpha$ are replaced by the functions
$z = re^{i\theta} \mapsto r^\alpha e^{i\theta}$. This is precisely the
power function $\Psi_{\alpha, 1}$. More generally, in the framework
developed in Section \ref{SSmult}, it is natural to generalize Zhan's
question by asking for which values of $\alpha, \beta$ does
$\Psi_{\alpha, \beta}$ preserve positivity when applied entrywise. Our
third result, Theorem \ref{Tnew}, provides bounds on $\alpha, \beta$
which guarantee that $\Psi_{\alpha, \beta}$ preserves or does not
preserve Loewner positivity. 

\begin{utheorem}\label{Tnew}
Let $n \geq 3$.
\begin{enumerate}
\item The entrywise function $\Psi_{\alpha,\beta}$ preserves Loewner
positivity on $\bp_n(\C)$ if $\beta \in \Z$, $(\alpha, \beta) \ne (0,0)$,
and either $\alpha \in |\beta| -2 +2\N$ or $\alpha \geq \max(n-2,|\beta|
+ 2n-6)$.
\item The entrywise function $\Psi_{\alpha,\beta}$ fails to preserve
positivity if either:
\begin{enumerate}
\item $\beta \not\in \Z$, or
\item $\alpha < 1$, or 
\item $1 \leq \alpha < \max(n-2,|\beta| + 2 \lfloor (\sqrt{8n+1}-5)/2
\rfloor)$ and $\alpha \not\in |\beta| -2 +2\N$. 
\end{enumerate}
\end{enumerate}
\end{utheorem}

\noindent Thus for $n \geq 3$, $\beta \in \Z$, and $\alpha \not\in
|\beta| -2 +2\N$, we see that $\Psi_{\alpha,\beta}$ preserves Loewner
positivity for $\alpha \geq \max(n-2, |\beta| + 2n-6)$, but not for
$\alpha < \max(n-2, |\beta| + 2 \lfloor (\sqrt{8n+1}-5)/2 \rfloor)$. Note
that if $n=3$, these two quantities coincide and equal $\max(1,|\beta|)$.
We therefore have the following corollary, which completely answers
Zhan's question for the $n=3$ case.

\begin{corollary}\label{Czhan}
For $n=3$, the entrywise power function $\Psi_{\alpha,\beta}$ preserves
Loewner positivity on $\bp_n(\C)$ if and only if $\beta \in \Z$ and
$\alpha \geq \max(1,|\beta|)$.
\end{corollary}

\noindent A consequence of Theorem \ref{Tnew} is that complex critical
exponents exist for the power functions $\Psi_{\alpha,\beta}$: 

\begin{corollary}\label{Ccrit}
For every $n \geq 3$ and $\beta \in \Z$, there exists a smallest real
number $\alpha_{\min}$ such that $\Psi_{\alpha,\beta}[-]$ preserves
$\bp_n(\C)$ for all $\alpha \geq \alpha_{\min}$. Moreover, $\alpha_{\min}
= \max(1,|\beta|)$ for $n=3$, while for $n \geq 4$,
\[ \max(n-2,|\beta| + 2 \lfloor (\sqrt{8n+1}-5)/2 \rfloor) \leq
\alpha_{\min} \leq |\beta| + 2n-6. \]
\end{corollary}

\noindent Note that Theorem \ref{TFitzHorn} and an application of the
Schur product theorem imply that $n-2 \leq \alpha_{\min} \leq |\beta| +
2n-4$. Corollary \ref{Ccrit} thus greatly improves this lower bound for
the critical exponent $\alpha_{\min}$. 

\section{Powers preserving positivity: the block case}\label{Sblock}

We now characterize powers preserving positivity when applied blockwise.
To prove Theorem \ref{Tmain1} we need some preliminaries. First recall
the notion of an $m$-matrix monotone function. 

\begin{definition}
Let $I \subset \R$ be an interval and let $m \geq 1$. A function $f: I
\to \R$ is said to be {\it $m$-matrix monotone} (or {\it $m$-monotone})
if given $m \times m$ Hermitian matrices $A,B$ with spectrum in $I$,
\[ A \geq B \implies f(A) \geq f(B). \]
\end{definition}

\noindent The following lemma reformulates $m$-monotonicity of power
functions in terms of block matrices, and will be crucial in proving
Theorem \ref{Tmain1}.

\begin{lemma}\label{Lmonotone}
Given an integer $m \in \N$, define the subset $\calp_m \subset
\bp_{2m}(\C)$ via:
\[ \calp_m := \{ \begin{pmatrix} A & B\\ B & C \end{pmatrix} \in
\bp_{2m}^{[m]}([0,\infty)) : \det C \neq 0, BC = CB \}. \]

\noindent Also fix $\alpha \in \R$. Then the following are equivalent:
\begin{enumerate}
\item The blockwise power function $f_\alpha^{[m]}[-]$ sends $\calp_m$ to
$\bp_{2m}(\C)$.

\item The function $f_\alpha$ is $m$-monotone on $(0,\infty)$.
\end{enumerate}

\noindent In particular, if $f_\alpha^{[m]}[-]$ preserves Loewner
positivity on $\bp_{mn}^{[m]}(\C)$ for some $n \geq 2$, then it is
$m$-monotone.
\end{lemma}

\begin{proof}
First suppose $f_\alpha^{[m]}[-]$ preserves Loewner positivity on
$\calp_m$, and assume $A \geq B > 0$. Let $X \in \bp_m(\C)$ denote the
principal square root of $B$. Then the block matrix  $M :=
\begin{pmatrix} A & X \\ X & I_m \end{pmatrix} \in
\bp_{2m}^{[m]}([0,\infty))$, by computing the Schur complement of $I_m$
in $M$. Therefore by hypothesis, the matrix $f_\alpha^{[m]}[M] =
\begin{pmatrix} A^\alpha & X^\alpha \\ X^\alpha & I_m \end{pmatrix}$ is
also positive semidefinite. Using Schur complements again, we conclude
that $A^\alpha - (X^\alpha)^2 = A^\alpha - B^\alpha \geq 0$. Thus $A \geq
B > 0 \Rightarrow A^\alpha \geq B^\alpha$ and so $f_\alpha$ is
$m$-monotone on $(0,\infty)$.

Conversely, suppose $f_\alpha$ is $m$-monotone on $(0,\infty)$, and
suppose $\begin{pmatrix} A & B \\ B & C \end{pmatrix} \in \calp_m$. Then
$A \geq B C^{-1} B$ (by taking Schur complements). Moreover, $B,C$ are
simultaneously diagonalizable, whence $B, C^{\pm 1}$ commute. It is now
easy to verify that $(B C^{-1} B)^\alpha = B^\alpha (C^\alpha)^{-1}
B^\alpha$. Now using the $m$-monotonicity of $f_\alpha$, we compute:
\[ C^\alpha \geq 0^\alpha = 0, \qquad A^\alpha = f_\alpha(A) \geq
f_\alpha(B C^{-1} B) = B^\alpha (C^\alpha)^{-1} B^\alpha. \]

\noindent In turn, this implies that the matrix $\begin{pmatrix} A^\alpha
& B^\alpha\\ B^\alpha & C^\alpha \end{pmatrix}$ is positive semidefinite,
proving (1). The final assertion is also clear since $\calp_m \oplus {\bf
0}_{m(n-2) \times m(n-2)} \subset \bp_{mn}^{[m]}(\C)$ (via padding by
zeros).
\end{proof}

Matrix monotone functions have been the subject of a detailed analysis by
Loewner \cite{Loewner34} and many others including Wigner and von Neumann
\cite{Wigner_et_al_54}, Bendat and Sherman \cite{Bendat_et_al_55},
Kor\'anyi \cite{Koranyi_56}, Donoghue \cite{Donoghue_74}, Sparr
\cite{Sparr_80}, Hansen and Petersen \cite{Hansen_et_al_81}, Ameur
\cite{Ameur_2003}, and more recently by Hansen \cite{Hansen2013} - also
see \cite{Hansen2013} for a history of the problem.
We now state an important and interesting characterization of matrix
monotone functions using Loewner matrices. This result was shown by
Hansen \cite{Hansen2013} and plays an essential role in proving Theorem
\ref{Tmain1}. 

\begin{definition}
Let $I \subset \R$ and $f : I \to \R$ be differentiable. The {\it first
divided difference} of $f$ for $\lambda_1, \lambda_2 \in I$, denoted by
$[\lambda_1, \lambda_2]_f$ is given by 
\[ [\lambda_1, \lambda_2]_f := \begin{cases}
\frac{f(\lambda_1)-f(\lambda_2)}{\lambda_1 - \lambda_2} & \textrm{if }
\lambda_1 \ne \lambda_2, \\
f'(\lambda_1) & \textrm{if } \lambda_1 = \lambda_2.
\end{cases} \]

\noindent Now given $m \geq 2$ and $\lambda_1, \dots, \lambda_m \in I$,
define the {\it Loewner matrix} $L_f(\lambda_1, \dots, \lambda_m)$ of $f$ at
the points $\lambda_j$ to be
\begin{equation}
L_f(\lambda_1, \dots, \lambda_m) := ([\lambda_s,\lambda_t]_f)_{s,t=1}^m.
\end{equation}
\end{definition}

\begin{theorem}[{Hansen \cite[Theorem 3.2]{Hansen2013}}]\label{THansen}
Let $m \in \N$ and $f$ be a real function in $C^1(I)$, where $I \subset
\R$ is an open interval. Then $f$ is $m$-monotone if and only if the
Loewner matrix $L_f(\lambda_1 , \dots, \lambda_m)$ is positive
semidefinite for all sequences $\lambda_1, \dots, \lambda_m \in I$.
\end{theorem}

\noindent We now have all the ingredients for proving Theorem
\ref{Tmain1}.

\begin{proof}[{\bf Proof of Theorem \ref{Tmain1}}]\hfill

\noindent {\bf Proof of (1).} 
Clearly, $f_1^{[m]}[-]$ preserves positivity on $\bpm_{mn}([0,\infty))$.
Next, if $\alpha = 0$ and the blocks $H_{st}$ are positive definite, then
$f_0^{[m]}[(H_{st})] = {\bf 1}_{n \times n} \otimes I_m$, where $\otimes$
denotes the Kronecker product, and so $f_0^{[m]}[(H_{st})] \in
\bp_{mn}(\C)$. Now assume $\alpha \in \R$ and $\alpha \ne 0, 1$. We claim
that the function $f_\alpha^{[m]}[-]$ does not preserve positivity on
$\bpm_{mn}([0,\infty))$. It suffices to prove the claim for $m=n=2$ (the
general case follows by padding with zeros).

Thus, suppose $f_\alpha^{[2]}[-]$ preserves positivity on
$\bp_4^{[2]}((0,\infty))$. By Lemma \ref{Lmonotone}, the function
$f_\alpha(x) = x^\alpha$ is $2$-monotone on $(0,\infty)$. By Theorem
\ref{THansen}, this is possible if and only if the Loewner matrix 
$L_{f_\alpha}(\lambda_1, \lambda_2)$
is positive semidefinite for all $\lambda_1, \lambda_2 > 0$
such that $\lambda_1 \ne \lambda_2$. Thus, the $(1,1)$-entry of
$L_{f_\alpha}(\lambda_1,\lambda_2)$ has to be nonnegative and so $\alpha
\geq 0$. Computing the determinant of $L_{f_\alpha}(\lambda_1,
\lambda_2)$, we obtain:
\begin{equation}
\det L_{f_\alpha}(\lambda_1, \lambda_2) = \alpha \lambda_1^{\alpha-1}
\cdot \alpha \lambda_2^{\alpha-1} - \left(\frac{\lambda_1^\alpha -
\lambda_2^\alpha}{\lambda_1-\lambda_2}\right)^2 \geq 0 \qquad \forall
\lambda_1, \lambda_2 > 0, \lambda_1 \ne \lambda_2. 
\end{equation}

\noindent Now fix $\lambda_2 > 0$. If $\alpha > 1$, then $\det
L_{f_\alpha}(\lambda_1, \lambda_2) \to -\infty$ as $\lambda_1 \to \infty$
since $\alpha \ne 0,1$. Thus, $\det L_{f_\alpha}(\lambda_1, \lambda_2) <
0$ for $\lambda_1$ large enough. This proves that $f_\alpha(x) =
x^\alpha$ is not $2$-monotone, and hence $f_\alpha^{[m]}[-]$ does not
preserve positivity if $\alpha > 1$ or $\alpha < 0$.

Finally, suppose $\alpha \in (0,1)$. We first claim that there exists a
real matrix $\begin{pmatrix} A & X \\ X & N \end{pmatrix} \in
\bp_4^{[2]}((0,\infty))$ such that the matrix $\begin{pmatrix} A^\alpha &
X^\alpha \\ X^\alpha & N^\alpha \end{pmatrix}$ is not positive
semidefinite. To prove the claim, consider the matrix
\begin{equation}\label{E44}
M := \begin{pmatrix} 3/2 & 0 & 1 & 1/2\\ 0 & 2 & 1/2 & 1\\ 1 & 1/2 & 1 &
4/5\\ 1/2 & 1 & 4/5 & 223/250 \end{pmatrix} = \begin{pmatrix} A & X \\ X
& N \end{pmatrix},
\end{equation}

\noindent where $A,X,N \in \bp_2(\R)$. It can be verified that $\det
(\lambda I_4- M)$ is a fourth-degree polynomial which is positive for
$|\lambda|$ large and at $\lambda = 1,4$; zero at $\lambda = 0$; and
negative at $1/5, 2$. Therefore $0$ is an eigenvalue of $M$, and the
other three eigenvalues of $M$ lie in $(1/5,1), (1,2), (2,4)$. It is now
easily verified that $M \in \bp_4^{[2]}((0,\infty))$.
We next claim that $f_\alpha^{[2]}[M] \notin \bp_4$ for small $\alpha >
0$ close enough to zero. To verify the claim, we will compute explicitly
the determinant of $f_\alpha^{[2]}[M]$, and show that it is negative
close to $\alpha = 0$.
We begin by computing the powers of the $2 \times 2$ blocks $A,X,N$ of
$M$. The block $A$ is diagonal, while the powers of the off-diagonal
block $X$ are computed using its spectral decomposition:
\[ X = \begin{pmatrix} 1 & \frac{1}{2}\\ \frac{1}{2} & 1\end{pmatrix} = U
\diag(\frac{1}{2}, \frac{3}{2}) U^T, \ U := \frac{1}{\sqrt{2}}
\begin{pmatrix} -1 & 1\\ 1 & 1\end{pmatrix}\]
from which it follows that
\[
X^\alpha = \frac{1}{2} \begin{pmatrix} (3/2)^\alpha +
(1/2)^\alpha & (3/2)^\alpha - (1/2)^\alpha\\ (3/2)^\alpha - (1/2)^\alpha
& (3/2)^\alpha + (1/2)^\alpha \end{pmatrix}. \]

To compute the spectral powers of the last remaining block $N :=
\begin{pmatrix} 1 & 4/5\\ 4/5 & 223/250 \end{pmatrix}$, we define $x_\pm
:= 27 \pm \sqrt{160729} = 27 \pm \sqrt{27^2 + 400^2}$ for convenience.
Then $N$ has spectral decomposition $N = V D V^{-1}$, where
\[ V := \begin{pmatrix} x_-/400 & x_+/400 \\ 1 & 1 \end{pmatrix}, \quad
D := \diag(1 - \frac{x_+}{500}, 1 - \frac{x_-}{500}), \quad
V^{-1} = \frac{1}{2 \sqrt{160729}} \begin{pmatrix} -400 & x_+\\
400 & -x_- \end{pmatrix}. \]

\noindent Let $\lambda_\pm := 1 - \frac{x_\pm}{500}$ be the eigenvalues
of $N$. Since $V = U D'$ with $U$ unitary and $D'$ diagonal, we obtain:
\[ N^\alpha := V D^\alpha V^{-1} = \frac{1}{2\sqrt{160729}}
\begin{pmatrix}x_+ \lambda_-^\alpha - x_- \lambda_+^\alpha &
400(\lambda_-^\alpha - \lambda_+^\alpha) \\
400(\lambda_-^\alpha - \lambda_+^\alpha) & x_+ \lambda_+^\alpha - x_-
\lambda_-^\alpha \end{pmatrix}. \]

\noindent Therefore if we define
$g_M(\alpha) := \det f_\alpha^{[2]}[M]$, then
\begin{align*}
\frac{4}{2^\alpha} g_M(\alpha) &= 2^{2-\alpha} \det \begin{pmatrix}
A^\alpha & X^\alpha \\ X^\alpha & N^\alpha \end{pmatrix}
= \ 4 a^2 b^3 + 4 a L_- L_+ + \frac{54}{\sqrt{160729}} ab(1-ab)(L_- -
L_+)\\
&\ + (ab+1) \left( (2L-1)(L_- a^2 + L_+ b^2) - (2L+1)(L_+ a^2 + L_- b^2)
\right),
\end{align*}

\noindent where $L_\pm := \lambda_\pm^\alpha, a := (3/2)^\alpha, b :=
(1/2)^\alpha$, and $L := 200 / \sqrt{160729}$.
Note that $g_M(0) = \det f_0^{[2]}[M] = 0$. Moreover, using the explicit
form of the function $g_M(\alpha)$, it can be verified that $g_M'(0) = 0$
and $g_M''(0) < 0$. This shows that $g_M(\alpha) < 0$ for all $0 <
|\alpha| < \epsilon_M$ for some $\epsilon_M >0$.

Now suppose $f_\alpha^{[2]}[-]$ preserves positivity on
$\bp_4^{[2]}((0,\infty))$ for some $\alpha \in (0,1)$.
Choose $k \in \N$ such that $\alpha^k \in (0,\epsilon_M)$, with
$\epsilon_M$ as above. Then $(f_\alpha^{[2]})^{\circ k}[M] =
f_{\alpha^k}^{[2]}[M] \in \bp_4^{[2]}((0,\infty)) \subset \bp_4(\C)$,
which contradicts the previous paragraph. This proves that
$f_\alpha^{[2]}[-]$ does not preserve positivity for $\alpha \in
(0,1)$.\medskip

\noindent {\bf Proof of (2).}
The first part shows that $\phi_\alpha^{[m]}[-]$ does not preserve
positivity on $\bpm_{mn}(\R)$ for $\alpha \ne 0, 1$. We now prove that
$\phi_\alpha^{[m]}[-]$ also does not preserve positivity for $\alpha = 0$
and $\alpha = 1$. Suppose first $\alpha = 0$. Fix
$B := \begin{pmatrix}0 & 0 \\ 1 & 1 \end{pmatrix}$, and for $c \in \R$,
define the matrix
\begin{equation}\label{eqn:phi_0}
A(c) := \begin{pmatrix}
c I_2 & B \\
B^T & c I_2
\end{pmatrix}.
\end{equation}

\noindent Note that $A(c)$ has eigenvalues $c, c, c \pm \sqrt{2}$.
Moreover, $B$ is diagonalizable and has eigenvalues $0$ and $1$. As a
consequence, $\phi_0(B) = B$. Therefore the matrix $A(\sqrt{2}) \in
\bp_4^{[2]}(\R)$, but $\phi_0^{[2]}[A(\sqrt{2})] = A(1) \not\in \bp_4$.
This proves $\phi_0^{[2]}[-]$ does not preserve positivity on
$\bp_{4}^{[2]}(\R)$.
The case of general $m,n \geq 2$ follows by padding $A(\sqrt{2})$ with
zeros. To prove that $\phi_1^{[m]}[-]$ does not preserve positivity on
$\bpm_{mn}(\R)$, consider the matrix
\begin{equation}
M := \begin{pmatrix}
2 & 0 & -1 & -1 \\
0 & 1 & -1 & 0 \\
-1 & -1 & 2 & 0 \\
-1 & 0 & 0 & 1
\end{pmatrix}
\end{equation} 

\noindent It is not difficult to verify that $M \in \bp_{4}^{[2]}(\R)$,
but $\det \phi_1^{[2]}[M] = -4/5$. This proves that $\phi_1^{[2]}[-]$
does not preserve positivity on $\bp_4^{[2]}(\R)$. It follows that
$\phi_1^{[m]}[-]$ does not preserve positivity on
$\bp_{mn}^{[m]}(\R)$ for $m,n \geq 2$.\medskip

\noindent {\bf Proof of (3).}
By part (1), the function $\psi_\alpha^{[m]}[-]$ does not preserve
positivity if $\alpha \ne 0, 1$. Clearly, $\psi_1^{[m]}[-]$ preserves
positivity since $\psi_1(x) = x$ for all $x \in \R$. That
$\psi_0^{[m]}[-]$ does not preserve positivity on $\bp_{mn}^{[m]}(\R)$
follows by considering the matrix $A(c)$ in Equation \eqref{eqn:phi_0}.
\medskip

\noindent {\bf Proof of (4).} 
By part (1), $\Psi_{\alpha,\beta}^{[m]}[-]$ does not preserve positivity
on $\bpm_{mn}(\C)$ if $\alpha \ne 0,1$. Moreover, the above analysis of
the matrix $A(c)$ in Equation \eqref{eqn:phi_0} shows that $\Psi_{0,
\beta}^{[m]}[-]$ does not preserve positivity on $\bpm_{mn}(\C)$ for any
$\beta \in \Z$. Now suppose $\alpha = 1$. By the second part of the
proof, $\Psi_{1,0}^{[m]} \equiv \phi_1^{[m]}$ does not preserve
positivity on $\bpm_{mn}(\C)$. Also, $\Psi_{1,1}^{[m]}$ clearly preserves
positivity. Note that since a matrix $A$ is positive semidefinite if and
only if its complex conjugate $\overline{A}$ is positive semidefinite,
$\Psi_{\alpha,\beta}^{[m]}[-]$ preserves positivity on $\bpm_{mn}(\C)$ if
and only if $\Psi_{\alpha,-\beta}^{[m]}[-]$ does so. To conclude the
proof, it thus remains to prove that $\Psi_{1,\beta}^{[m]}[-]$ does not
preserve positivity on $\bpm_{mn}(\C)$ for $\beta \geq 2$. Without loss
of generality, let $m = n = 2$, and define:
\begin{equation}\label{EMabc}
M(a,b,c) := \begin{pmatrix}
1 & 0 & a & b \\
0 & 1 & c & a \\
\overline{a} & \overline{c} & 1 & 0 \\
\overline{b} & \overline{a} & 0 & 1
\end{pmatrix} \qquad a,b,c \in \C. 
\end{equation}

\noindent One verifies that the four eigenvalues of the matrix $M(a,a,0)$
are $1 \pm a(\sqrt{5} \pm 1)/2$. Therefore if we fix $a \in (0,
(\sqrt{5}-1)/2)$, the matrix $M(a,a,0)$ is positive definite.
Consequently, there exists $\epsilon > 0$ such that $M(a,a,c) \in
\bp_4^{[2]}((0,\infty))$ for $|c| < \epsilon$.

We now claim that $\Psi_{1,\beta}^{[2]}[M(a,a,c)] \not\in \bp_4(\C)$
if $c$ is negative and close enough to $0$. To prove the claim, we
first compute $\Psi_{1,\beta}^{[2]}[M(a,a,c)]$.
Note that $\Psi_{1,\beta}(I_2) = I_2$; now set $B := \begin{pmatrix} a &
a \\ c & a \end{pmatrix}$, with $c < 0$. The eigenvalues of $B$ are $a
\pm i \sqrt{a |c|}$, with corresponding eigenvectors $v_\pm := (\mp i
\sqrt{a/|c|}, 1)^T$. As a consequence,
defining $\lambda_\pm := \Psi_{1,\beta}(a \pm i \sqrt{a |c|})$, we
obtain:
\begin{align*}
\Psi_{1,\beta}(B) &= \begin{pmatrix} -i \sqrt{a/|c|} & i \sqrt{a/|c|} \\
1 & 1\end{pmatrix} \begin{pmatrix}\lambda_+ & 0 \\ 0 &
\lambda_-\end{pmatrix} \begin{pmatrix} -i \sqrt{a/|c|} & i \sqrt{a/|c|}
\\ 1 & 1\end{pmatrix}^{-1} \\
&= \begin{pmatrix}\frac{\lambda_+ + \lambda_-}{2} & \frac{-i
\sqrt{a}}{\sqrt{|c|}} \cdot \frac{\lambda_+ - \lambda_-}{2} \\ \frac{i
\sqrt{|c|}}{\sqrt{a}} \cdot \frac{\lambda_+ - \lambda_-}{2} &
\frac{\lambda_+ + \lambda_-}{2}\end{pmatrix} = \begin{pmatrix} a' & b'\\
c' & a' \end{pmatrix},
\end{align*}

\noindent say. Thus $\Psi_{1,\beta}^{[2]}[M(a,a,c)] = M(a',b',c')$.

Now suppose $\beta \geq 2$. We will prove that there exists $a > 0$ such
that the $(1,2)$-entry of the real matrix $\Psi_{1,\beta}(B)$ is greater
than $1$, if $c$ is negative and close enough to $0$. Indeed, note that
\[ \lambda_\pm = \Psi_{1,\beta}(a \pm i \sqrt{a |c|}) = \sqrt{a^2 + a
|c|} e^{i\beta \arctan(\pm \sqrt{|c|/a})}. \]

\noindent Thus, 
\begin{align*}
\lim_{c \to 0^-} \Psi_{1,\beta}(B)_{12} &= \lim_{c \to 0^-} \frac{-i
\sqrt{a}}{\sqrt{|c|}} \cdot \frac{\lambda_+ - \lambda_-}{2}\\
&= \lim_{c \to 0^-} \frac{-i \sqrt{a}}{\sqrt{|c|}} \sqrt{a^2 + a |c|}
\frac{e^{i\beta \arctan(\sqrt{|c|/a})} - e^{i\beta
\arctan(-\sqrt{|c|/a})}}{2} \\
&= -i a \lim_{c \to 0^-} \frac{e^{i\beta \arctan(\sqrt{|c|/a})} -
e^{i\beta \arctan(-\sqrt{|c|/a})}}{2\sqrt{|c|/a}} \\
&= -i a \left.\frac{d}{dy} e^{i\beta \arctan(y)}\right|_{y = 0} = -i a
\left.e^{i\beta \arctan(y)} i \beta \frac{1}{1+y^2}\right|_{y=0} = a
\beta.
\end{align*}

\noindent As a consequence, if $\beta \geq 2$ and $a \in (1/\beta,
(\sqrt{5}-1)/2)$, then for $c < 0$ small enough, the $(1,2)$-entry of
$\Psi_{1,\beta}(B)$ is greater than $1$.
But then the minor of $\Psi_{1,\beta}^{[2]}[M(a,a,c)]$ obtained by
deleting the second row and column is negative, from which it follows
that $\Psi_{1,\beta}^{[2]}[M(a,a,c)] \not\in \bp_4(\C)$.
Therefore $\Psi_{1,\beta}^{[2]}[-]$ does not preserve positivity on
$\bp_4^{[2]}(\C)$ if $\beta \ne \pm 1$. As before, the case of general
$m,n \geq 2$ follows by padding with zeros. This concludes the proof.
\end{proof}

\begin{remark}
In the proof of part (1) of Theorem \ref{Tmain1}, we showed that $\det
f_\alpha^{[2]}[M] < 0$ for all $\alpha \in (0,\epsilon_M)$ for some
$\epsilon_M \in (0,1)$, with $M$ as in Equation \eqref{E44}. In fact,
numerical computations indicate that $\det f_\alpha^{[2]}[M] < 0$ for all
$\alpha \in (0,1)$; this would provide a ``universal" counterexample $M$
for the proof of part (1).
\end{remark}

\section{Powers preserving positivity for commuting blocks}\label{Scommute}

Recall that D.~Choudhury \cite{Dipa_proc} studied an interesting variant
of the problem considered in Section \ref{Sblock} - namely, which
blockwise powers $(H_{st})_{s,t=1}^n \mapsto (H_{st}^\alpha)$ preserve
positivity when all the $m \times m$ blocks $H_{st}$ commute and are
positive semidefinite. It was shown in \cite{Dipa_proc} that if $\alpha
\in \N \cup [mn-2,\infty)$ then the corresponding blockwise power
preserves positivity. We now demonstrate that the bound $mn-2$ can be
significantly improved. More precisely, we completely characterize the
powers preserving Loewner positivity in that setting. 

\begin{proof}[{\bf Proof of Theorem \ref{Tmain2}}]
The proof is a refinement of the argument in \cite[Theorem 5]{Dipa_proc}.
Let $H = (H_{st}) \in \bp_{mn}(\C)$ be as given. Since the blocks
$H_{st}$ commute, they are simultaneously diagonalizable, i.e., there
exists a $m \times m$ unitary matrix $U$ and diagonal matrices
$\Lambda_{st}$ such that $H_{st} = U \Lambda_{st} U^* \ \forall s,t$.
Letting $T := U^{\oplus n}$ and $\Lambda := (\Lambda_{st})$, we obtain
$H = T \Lambda T^{-1}$. Let $P$ be the permutation matrix such that 
\begin{equation}\label{Eperm}
P^{-1} \Lambda P = A_1 \oplus \dots \oplus A_m,
\end{equation} 

\noindent where $(A_{k})_{st} := (\Lambda_{st})_{kk}$ with $1 \leq k \leq
m$ and $1 \leq s,t \leq n$. Then $H = (TP) (A_1 \oplus \dots \oplus A_m)
(TP)^{-1}$. By assumption, $A_k \in \bp_n([0,\infty))\ \forall k$.
Moreover, since the entries of the matrices $A_k$ are the eigenvalues of
the blocks $H_{st}$, we have 
$(H_{st}^\alpha) = (TP) (A_1^{\circ \alpha} \oplus \dots \oplus
A_m^{\circ \alpha})(TP)^{-1}$.
Here $A^{\circ \alpha} := (a_{st}^\alpha)$ denotes the entrywise power of
$A = (a_{st})$. Since $A_k$ are $n \times n$ matrices, it follows
immediately by Theorem \ref{TFitzHorn} that $(H_{st}^\alpha) \in
\bp_{mn}(\C)$ if $\alpha \in \N \cup [n-2,\infty)$.

Now suppose $\alpha \in (0,n-2) \setminus \N$. Choose $\epsilon > 0$ such
that the matrix $A := (1 + \epsilon st)_{s,t=1}^n$ satisfies $A^{\circ
\alpha} \not\in \bp_n$ (see Theorem \ref{TFitzHorn}). Let 
\begin{equation}\label{Econst1}
\Lambda = (\Lambda_{st})_{s,t=1}^n := P A^{\oplus m} P^{-1},
\end{equation} 

\noindent where $P$ is the permutation matrix given in Equation
\eqref{Eperm} and $\Lambda_{st}$ are $m \times m$ diagonal matrices.
Define $H_{st} := \Lambda_{st}$. Then the matrices $H_{st}$ are Hermitian
positive semidefinite, as is the matrix $H = (H_{st})$, but
$(H_{st}^\alpha) = P (A^{\circ \alpha} \oplus \dots \oplus A^{\circ
\alpha}) P^{-1}$ is not positive semidefinite by construction of $A$.
This shows that the powers $\alpha \in (0,n-2) \setminus \N$ do not
preserve positivity when applied blockwise. 

Finally, suppose $\alpha < 0$. Let $A := I_{m \times m} + {\bf 1}_{m
\times m} \in \bp_m([1,2])$. Examining the leading principal $2 \times 2$
block of $A$, it follows that $A^{\circ \alpha} \not\in \bp_m$.
Repeating the same construction as in Equation \eqref{Econst1}, we
conclude that there exist commuting blocks $H_{st} := \Lambda_{st} \in
\bp_m(\C)$ such that $(H_{st}) \in \bp_{mn}(\C)$, but $(H_{st}^\alpha)
\not\in \bp_m(\C)$ if $\alpha < 0$. This concludes the proof. 
\end{proof}

In Theorem \ref{Tmain2}, we assumed each block $H_{st}$ to be positive
semidefinite. This assumption was necessary for the powers
$H_{st}^\alpha$ to be well-defined. We now consider the case where the
blocks are not positive semidefinite. Using the functions $\phi_\alpha$
and $\psi_\alpha$, it is natural to extend the characterization provided
by Theorem \ref{Tmain2} to Hermitian blocks with arbitrary eigenvalues.
Using Theorem \ref{Tcrit}, we can now characterize the powers $\alpha$
such that $\phi_\alpha^{[m]}$ and $\psi_\alpha^{[m]}$ preserve positivity
when the blocks commute.

\begin{theorem}\label{Tmain3}
Let $\alpha \in \R \setminus \{0\}$ and $m,n \geq 2$. Then 
\begin{enumerate}
\item $\phi_\alpha^{[m]}[H] \in \bp_{mn}(\C)$ for all $m \times m$
Hermitian matrices $H_{st}$ such that $(H_{st}) \in \bp_{mn}(\C)$ and the
blocks $H_{st}$ commute if $\alpha \in 2\N \cup [n-2,\infty)$. If $\alpha
\not\in 2\N \cup [n-2,\infty)$, there exist real symmetric matrices
$H_{st}$ such that $(H_{st}) \in \bp_{mn}(\R)$, the blocks $H_{st}$
commute, but $\phi_\alpha^{[m]}[H] \not\in \bp_{mn}(\R)$. 

\item $\psi_\alpha^{[m]}[H] \in \bp_{mn}(\C)$ for all Hermitian $m \times
m$ matrices $H_{st}$ such that $(H_{st}) \in \bp_{mn}(\C)$ and the blocks
$H_{st}$ commute if $\alpha \in (-1+2\N) \cup [n-2,\infty)$. If $\alpha
\not\in (-1+2\N) \cup [n-2,\infty)$, there exist real symmetric matrices
$H_{st}$ such that $(H_{st}) \in \bp_{mn}(\R)$, the blocks $H_{st}$
commute, but $\psi_\alpha^{[m]}[H] \not\in \bp_{mn}(\R)$. 
\end{enumerate}
\end{theorem}

\begin{proof}
The proof is similar to the proof of Theorem \ref{Tmain2}. 
Let $U$ be a unitary matrix and $P$ be a permutation matrix such that
defining $H := (H_{st})$ and $T := U^{\oplus n}$,  we have
\begin{equation}\label{Efactor}
H =  (TP) (A_1 \oplus \dots \oplus A_m) (TP)^{-1}, 
\end{equation}

\noindent where $A_1, \dots, A_m$ are $n \times n$ matrices containing
the eigenvalues of the blocks $H_{st}$. If $f =
\phi_\alpha$ or $\psi_\alpha$, we have $f^{[m]}[H] = (TP) (f[A_1] \oplus \dots \oplus f[A_m]) (TP)^{-1}$. It follows from Theorem \ref{Tcrit} that
$\phi_\alpha^{[m]}[H] \in \bp_{mn}(\C)$ if $\alpha \in 2\N \cup
[n-2,\infty)$ and $\psi_\alpha^{[m]}[H] \in \bp_{mn}(\C)$ if $\alpha
\in (-1+2\N) \cup [n-2,\infty)$. Conversely, if $f = \phi_\alpha$ and
$\alpha \not\in 2\N \cup [n-2,\infty)$ or $f = \psi_\alpha$ and $\alpha
\not\in (-1+2\N) \cup [n-2,\infty)$, then by \cite[Theorem 2.5,
Proposition 6.2]{GKR-crit-2sided} there exists a matrix $A \in \bp_n$
such that $f[A] \not\in \bp_n$. Using the same construction as in
Equation \eqref{Econst1}, we conclude that $f^{[m]}[-]$ does not
preserve positivity. 
\end{proof}

\begin{remark}
We now address the case $\alpha = 0$, which was omitted from Theorem
\ref{Tmain3} for ease of exposition. We first claim that if $n = 2$ and
$H := (H_{st}) \in \bp_{2m}(\C)$ with Hermitian commuting blocks
$H_{st}$, then $\phi_0^{[m]}[H], \psi_0^{[m]}[H] \in \bp_{2m}(\C)$.
Indeed, as in Equation \eqref{Efactor}, the block matrix $H$ can be
factored as $H =  (TP) (A_1 \oplus \dots \oplus A_m) (TP)^{-1}$, where
$A_1, \dots, A_m\in \bp_2$. Moreover, $\phi_0, \psi_0$ preserve
positivity when applied entrywise to $\bp_2$, since the only possible
resulting matrices are ${\bf 0}_{2 \times 2}, {\bf 1}_{2 \times 2}$,
$I_{2 \times 2}$, and $\begin{pmatrix}1 & -1 \\ -1 & 1\end{pmatrix}$,
which are all positive semidefinite. However, when $n \geq 3$, we claim
that $\phi_\alpha^{[m]}[H], \psi_\alpha^{[m]}[H]$ are not always positive
semidefinite. Indeed, as in \cite[Equation 6.2]{GKR-crit-2sided}, define
\begin{equation}\label{Etop3}
A := \begin{pmatrix} 1 & 1/\sqrt{2} & 0\\ 1/\sqrt{2} & 1 &
1/\sqrt{2}\\ 0 & 1/\sqrt{2} & 1 \end{pmatrix} \oplus {\bf 0}_{(n-3)
\times (n-3)} \in \bp_n.
\end{equation}

\noindent One easily verifies that $\phi_0[A] = \psi_0[A] \not\in \bp_n$.
Using the same construction as in Equation \eqref{Econst1}, we conclude
that there exist commuting blocks $H_{st} := \Lambda_{st} \in \bp_m(\C)$
such that $H = (H_{st}) \in \bp_{mn}(\C)$, but $\phi_\alpha^{[m]}[H],
\psi_\alpha^{[m]}[H] \not\in \bp_{mn}(\C)$ when $\alpha = 0$. 
\end{remark}

\begin{remark}
An interesting consequence of Theorem \ref{Tmain3} is that when the
blocks commute, preserving positivity is in fact independent of the block
size $m$ (see part (2) of Theorem \ref{TdePillis}). This is in contrast
to Theorem \ref{Tmain1}, in which increasing the block size to $m \geq 2$
drastically reduces the set of powers preserving positivity, when the
commutativity assumption is omitted.
\end{remark}

\noindent {\bf Powers of the trace function.}
Problems similar to the ones above have been considered in the
literature, with the power function $H_{st} \mapsto H_{st}^\alpha$
replaced by other functions mapping $m \times m$ blocks to $p \times p$
matrices (see e.g.~\cite{Thompson_61, Marcus_Katz_69, depillis_69,
Marcus_watkins_71, Zhang_2012}). In particular, de Pillis
\cite{depillis_69} studies the map $(H_{st})_{s,t=1}^n \mapsto ({\rm
tr}(H_{st}))_{s,t=1}^n$ and demonstrates that it preserves positivity.
See also \cite{Zhang_2012} for a nice short proof of the same result. To
conclude this section, we extend de Pillis's result by characterizing the
values $\alpha \geq 0, \beta \in \Z$ such that $(H_{st}) \mapsto
(\Psi_{\alpha, \beta}(\tr(H_{st})))$ preserves positivity.

\begin{theorem}\label{TdePillis}
Fix $\alpha \geq 0$, $\beta \in \Z$, and $m,n \in \N$. Then the following
are equivalent: 
\begin{enumerate}
\item  $\Psi_{\alpha, \beta}[(\tr(H_{st}))_{s,t=1}^n] \in \bp_{n}(\C)$
for all $(H_{st})_{s,t=1}^n \in \bp_{mn}(\C)$.
\item $\Psi_{\alpha, \beta}[-]$ preserves positivity on $\bp_n(\C)$.
\item $\Psi_{\alpha, \beta}[(\tr(H_s^* H_t))_{s,t=1}^n] \in \bp_n(\C)$
for all $m \times m$ complex matrices $H_1, \dots, H_n$. 
\item $\Psi_{\alpha,\beta}^{[m]}[(H_{st})] \in \bp_{mn}(\C)$ if
$(H_{st})_{s,t=1}^n \in \bpm_{mn}(\C)$ and all blocks $H_{st}$ commute.
\end{enumerate}
\end{theorem}

\begin{proof}
Suppose first $(1)$ holds and let $A = (a_{st})_{s,t=1}^n \in \bp_n(\C)$.
Define $H_{st} \in \bp_m(\C)$ by $(H_{st})_{qr} := a_{st}$ if $q=r=1$ and
$0$ otherwise. Then $(H_{st})_{s,t=1}^n \in \bp_{mn}(\C)$, so
$\Psi_{\alpha, \beta}[A] \in \bp_n(\C)$ by (1). Thus $(1) \Rightarrow
(2)$. Conversely, if $(H_{st})_{s,t=1}^n \in \bp_{mn}(\C)$, then
$(\tr(H_{st}))_{s,t=1}^n \in \bp_n(\C)$ by \cite[Proposition
2.3]{depillis_69}, and $(2) \Rightarrow (1)$ follows immediately. Next,
$(2) \Leftrightarrow (3)$ because matrices of the form $(\tr(H_s^* H_t))$
are general Gram matrices in the inner product space $\C^{m \times m}$
with $\langle A, B \rangle := \tr(A^*B)$, so that the set of such
matrices coincides with $\bp_n(\C)$.
Finally, that $(2) \Leftrightarrow (4)$ follows by simultaneously
diagonalizing the blocks $H_{st}$ and proceeding as in the proof of
Theorem \ref{Tmain2}.  
\end{proof}

Note that when $\beta$ is even or odd, the function $\Psi_{\alpha,
\beta}$ reduces on $\R$ to $\phi_\alpha$ and $\psi_\alpha$ respectively.
Thus the powers $\alpha$ such that $\phi_\alpha[-]$ or $\psi_\alpha[-]$
preserves positivity on $\bp_n(\R)$ in Theorem \ref{TdePillis} are known
(see Theorem \ref{Tcrit}). In the next section, we explore the general
problem of characterizing the values $\alpha, \beta$ for which
$\Psi_{\alpha, \beta}[-]$ preserves Loewner positivity on $\bp_n(\C)$.

\section{Entrywise powers preserving positivity on Hermitian
matrices}\label{Sentrywise}

This section is devoted to proving Theorem \ref{Tnew}. As the proof is
long and intricate, we show the $n=3$ case in Section \ref{Sn=3}, and
then the general case in Section \ref{Sn_general}.

\subsection{Preserving positivity on Hermitian matrices of order
3}\label{Sn=3}

Note that for $n=1,2$, all maps $\Psi_{\alpha,\beta}$ preserve positivity
when applied entrywise to every matrix in $\bp_n(\C)$. In this subsection
we focus on the $n=3$ case. We begin by identifying a smaller sub-family
of matrices which it suffices to consider when verifying whether or not
$\Psi_{\alpha,\beta}$ preserves Loewner positivity.

\begin{lemma}\label{L3x3}
For $j=1,2,3$, suppose $r_j > 0, s_j \geq 0, t_j \in \R, \theta_j, \theta
\in (-\pi,\pi]$, and define ${\bf t} := (t_1, t_2, t_3)$. Now define:
\begin{equation}\label{EAp}
 A := \begin{pmatrix}
r_1 & s_3 e^{i \theta_3} & s_2 e^{i \theta_2} \\
s_3 e^{-i \theta_3} & r_2 & s_1 e^{i \theta 1} \\
s_2 e^{-i \theta_2} & s_1 e^{-i \theta_1} & r_3
\end{pmatrix}, \qquad
T({\bf t}, \theta) := \begin{pmatrix}
1 & t_3 & t_2 e^{i\theta} \\
t_3 & 1 & t_1 \\
t_2 e^{-i\theta} & t_1 & 1
\end{pmatrix}.
\end{equation}

\noindent Then the following are equivalent:  
\begin{enumerate}
\item $A \in \bp_3(\C)$; 
\item $T({\bf t},\theta) \in \bp_3(\C)$, where $t_j := \frac{s_j
\sqrt{r_j}}{\sqrt{r_1 r_2 r_3}}$ for $j = 1,2,3$, and $\theta = \theta_1
+ \theta_3 - \theta_2$. 
\item Given $t_j := \frac{s_j \sqrt{r_j}}{\sqrt{r_1 r_2 r_3}}$, we have
$t_j \in [0,1]$ for $j=1,2,3$, and $\det T({\bf t},\theta) = 1 -
\sum_{j=1}^3 t_j^2 + 2 t_1 t_2 t_3 \cos \theta \geq 0$. 
\end{enumerate}
\end{lemma}

\begin{proof}
Define $D:= \diag(r_1^{-1/2},r_2^{-1/2},r_3^{-1/2})$. That $(1)
\Leftrightarrow (2)$ follows from the fact that the principal minors of
$T({\bf t},\theta)$ are equal to the corresponding principal minors of
$DAD$, and hence are obtained from the principal minors of $A$ by
rescaling by positive factors. That $(2) \Leftrightarrow (3)$ is obvious. 
\end{proof}

The following corollary to Lemma \ref{L3x3} helps simplify the task of
ascertaining if an entrywise power function $\Psi_{\alpha,\beta}$
preserves Loewner positivity.

\begin{corollary}\label{Csp3}
Let $n \geq 3$, $\alpha \in \R$, and $\beta \in \Z$. Then
$\Psi_{\alpha,\beta}[-]$ preserves positivity on $\bp_3(\C)$ if and only
if $T({\bf t^{\circ \alpha}},\beta \theta) \in \bp_3(\C)$ for every ${\bf
t} \in [0,1]^3$ and $\theta \in (-3\pi, 3\pi)$ such that $\det T({\bf t},
\theta) \geq 0$. 
\end{corollary}

\begin{proof}
Clearly $\Psi_{\alpha,\beta}$ preserves positivity on $\bp_2(\C)$, hence
on matrices $A \in \bp_3(\C)$ with at least one zero diagonal entry. For
all other matrices $A \in \bp_3(\C)$, we are now done by Lemma
\ref{L3x3}. 
\end{proof}

In order to prove our next result, we recall the notion of a generalized
Dirichlet polynomial. 

\begin{definition}
A {\it generalized Dirichlet polynomial} is a function $F : \R \to \R$
of the form
$\displaystyle F(x) = \sum_{j=1}^n a_j t_j^x$,
where $a_j, t_j, x \in \R$ and $t_1 > t_2 > \cdots > t_n > 0$.
\end{definition}

Given a sequence $(a_j)_{j=1}^n$, denote by $S[(a_j)]$ the number of sign
changes in the sequence after discarding all zero terms $a_j$. Also
define $A_j := a_1 + \cdots + a_j$ for all $1 \leq j \leq n$. Then
$S[(A_j)] \leq S[(a_j)]$. We now recall the following classical result
which extends Descartes' Rule of Signs to generalized Dirichlet
polynomials. 

\begin{theorem}[Descartes' Rule of Signs,
\cite{Jameson,Laguerre}]\label{Tdescartes}
Suppose $F(x) = \sum_{j=1}^n a_j t_j^x : \R \to \R$ is a generalized
Dirichlet polynomial (with $t_1 > \cdots > t_n > 0$ as above), and $A_j =
a_1 + \cdots + a_j$ for all $j$. Then $F$ has at most $S[(A_j)]$ positive
zeros, and at most $S[(a_j)]$ real zeros.
\end{theorem}

Before we fully classify the entrywise powers which preserve Loewner
positivity on $\bp_3(\C)$, we first show that $\Psi_{\alpha,\beta}$
preserves positivity on $\bp_3(\C)$ if $\alpha \geq \max(1,|\beta|)$. We
also prove that $\Psi_{\alpha,\beta}$ does not preserve positivity on
$\bp_n(\C)$ if $\beta \not\in \Z$. In Section \ref{Sn_general}, we will
prove that $\Psi_{\alpha,\beta}$ does not preserve positivity on
$\bp_n(\C)$ if $\alpha < \max(n-2, |\beta| + 2 \lfloor (\sqrt{8n+1}-5)/2
\rfloor)$, thus completing the classification when $n=3$. 

\begin{theorem}\label{Tzhan}
For $n=3$, the entrywise power function $\Psi_{\alpha,\beta}$ preserves
Loewner positivity on $\bp_n(\C)$ if $\beta \in \Z$ and $\alpha \geq
\max(1,|\beta|)$. Moreover, if $\beta \not\in \Z$, then
$\Psi_{\alpha,\beta}$ does not preserve Loewner positivity on
$\bp_n(\C)$. 
\end{theorem}

\begin{proof}
Suppose $\beta \in \Z$ and $\alpha \geq \max(|\beta|,1)$. By Corollary
\ref{Csp3}, it suffices to show that $\Psi_{\alpha,\beta}$ preserves
positivity on all matrices $T({\bf t},\theta) \in \bp_3(\C)$ of the form
\eqref{EAp}. Using Lemma \ref{L3x3}, this reduces to showing:
\begin{equation}\label{E3x3}
1 - \sum_{j=1}^3 t_j^2 + 2 t_1 t_2 t_3 \cos \theta \geq 0 \quad \implies
\quad g_\beta(\alpha) := 1 - \sum_{j=1}^3 t_j^{2 \alpha} + 2 (t_1 t_2
t_3)^\alpha \cos(\beta \theta) \geq 0.
\end{equation}

\noindent In \eqref{E3x3} we may assume without loss of generality that
$\beta > 0$. There are now three cases: first, if $t_j = 0$ for some $j$,
then Equation \eqref{E3x3} is easy to show. Next, suppose $t_j$ are all
nonzero and $\max_j t_j = 1$, say $t_1 = 1$. Then $g_1(1) = -t_2^2 -t_3^2
+ 2 t_2 t_3 \cos \theta \geq 0$ if and only if $t_2 = t_3$ and $\cos
\theta = 1$. But then $\theta = 0$ or $\pm 2 \pi$ and \eqref{E3x3} again
follows. The third case is if $t_j \in (0,1)\ \forall j$. In this case we
use Theorem \ref{Tdescartes}: the partial sums of the coefficients are
$1, 0, -1, -2, -2 + 2 \cos(\beta \theta)$, and hence the generalized
Dirichlet polynomial has at most one positive root. First suppose
$\theta$ is not an integer multiple of $2\pi/\beta$.  Note that
$g_\beta(0) = 1 -3 + 2\cos(\beta \theta) < 0$. Also, by the Schur product
theorem, $g_\beta(\beta) \geq 0$ since $\beta \in \N$. Thus, the
generalized Dirichlet polynomial $g_\beta$ has a unique root between $0$
and $\beta$.
It follows that $g_\beta(\alpha) \geq 0$ for all $\alpha \geq \beta$,
since $g_\beta(\alpha) \to 1$ as $\alpha \to \infty$. Finally, suppose
$\theta = 2\pi k/\beta$ for some $k \in \Z$. To show \eqref{E3x3}, note
that
\begin{equation}
1 - \sum_{j=1}^3 t_j^2 + 2 t_1 t_2 t_3 \cos \theta \geq 0 \quad \implies
\quad 1 - \sum_{j=1}^3 t_j^2 + 2 t_1 t_2 t_3 \geq 0.
\end{equation}

\noindent This implies that the real matrix $T({\bf t},0)$ as in Equation
\eqref{EAp} is positive semidefinite. Now \eqref{E3x3} follows by
applying Theorem \ref{TFitzHorn} to $T({\bf t},0)$, since $\alpha \geq
1$.

To conclude the proof, we now provide a ``universal'' example of a matrix
$A \in \bp_3(\C)$ such that $\Psi_{\alpha,\beta}[A \oplus {\bf 0}_{(n-3)
\times (n-3)}] \not\in \bp_n(\C)$ whenever $\beta \in \R \setminus \Z$.
Define
\begin{equation}
A = \begin{pmatrix}
1 & e^{2\pi i/3} & e^{-2\pi i/3} \\
e^{-2\pi i/3} & 1 & e^{2\pi i/3} \\
e^{2\pi i/3} & e^{-2\pi i/3} & 1
\end{pmatrix}.
\end{equation}

\noindent Clearly $A \in \bp_3(\C)$, but $\det
\Psi_{\alpha,\beta}[A] = -2 + 2 \cos(2\pi \beta)$, which is negative
precisely when $\beta \not\in \Z$. Thus $\Psi_{\alpha,\beta}$ does not
preserve positivity on $\bp_n(\C)$ when $\beta \not\in \Z$.
\end{proof}

\subsection{Bounds for arbitrary dimension $n$}\label{Sn_general}

We now prove Theorem \ref{Tnew}, which addresses the case of general $n
\geq 3$. The proof will use the following preliminary result, which
generalizes an idea from FitzGerald and Horn \cite[Theorem
2.2]{FitzHorn}. 

\begin{proposition}\label{P01}
Let $\alpha > 1$ and fix an integer $n \geq 3$. Suppose
$\Psi_{\alpha-1,1}[A] \in \bp_{n-1}(\C)$ for all $A \in \bp_{n-1}(\C)$.
Then $\Psi_{\alpha,0}[A] \in \bp_n(\C)$ for all $A \in \bp_n(\C)$. 
\end{proposition}

\begin{proof}
Suppose $\Psi_{\alpha-1,1}[A] \in \bp_{n-1}(\C)$ for all $A \in
\bp_{n-1}(\C)$. Fix $z = z_1 + z_2 i,w = w_1 + w_2 i \in \C$, where
$z_1,z_2,w_1,w_2 \in \R$, and denote by $z_\lambda := \lambda z +
(1-\lambda)w$. Then 
\begin{align*}
\frac{d}{d\lambda} \Psi_{\alpha,0}(z_\lambda) = &\ \frac{\alpha}{2}
\Psi_{\alpha-2,0}(z_\lambda) \left[2(\lambda z_1 +
(1-\lambda)w_1)(z_1-w_1) + 2(\lambda z_2 +
(1-\lambda)w_2)(z_2-w_2)\right] \notag \\
= &\ \alpha \Psi_{\alpha-2,0}(z_\lambda) \ree(z_\lambda \overline{z-w}) =
\ \alpha \ree(\Psi_{\alpha-2,0}(z_\lambda) z_\lambda \overline{z-w}) \\
&= \alpha \ree(\Psi_{\alpha-1,1}(z_\lambda) \overline{z-w}). 
\end{align*}

\noindent We now proceed as in the proof of \cite[Theorem 2.2]{FitzHorn}.
Note that 
\begin{equation}
\Psi_{\alpha, 0}(z) = \Psi_{\alpha,0}(w) + \int_0^1 \frac{d}{d\lambda}
\Psi_{\alpha,0}(z_\lambda)\ d\lambda \ = \Psi_{\alpha,0}(w) + \alpha
\int_0^1 \ree(\Psi_{\alpha-1,1}(z_\lambda) \overline{z-w})\ d\lambda.
\label{eqn:horn_int}
\end{equation}

\noindent Now let $A \in \bp_n(\C)$ and let $\zeta := (a_{1n}, a_{2n},
\dots, a_{nn})^T /a_{nn}^{1/2}$ if $a_{nn} \ne 0$ and $\zeta := {\bf
0}_{n \times 1}$ otherwise. By \cite[Lemma 2.1]{FitzHorn}, the matrix $A
- \zeta \zeta^* \in \bp_n(\C)$. Also, note that the entries of the last
row and column of $A - \zeta \zeta^*$ are zero. Applying
\eqref{eqn:horn_int} entrywise, we obtain that 
\begin{equation}\label{eqn:horn_int2}
\Psi_{\alpha, 0}[A] = \Psi_{\alpha, 0}[\zeta \zeta^*] + \alpha \int_0^1
\ree\left(\Psi_{\alpha-1,1}[\lambda A + (1-\lambda) \zeta \zeta^*] \circ
\overline{A-\zeta \zeta^*}\right)\ d\lambda. 
\end{equation}

\noindent Note that the Schur product $\Psi_{\alpha-1,1}[\lambda A +
(1-\lambda) \zeta \zeta^*] \circ \overline{A-\zeta \zeta^*}$ in the
integrand in Equation \eqref{eqn:horn_int2} is positive semidefinite by
hypothesis and the fact that the last row and column of $A - \zeta
\zeta^*$ are zero. It follows immediately that $\Psi_{\alpha, 0}[A] \in
\bp_n(\C)$. This concludes the proof. 
\end{proof}

We now have all the ingredients necessary to prove our last main result.

\begin{proof}[{\bf Proof of Theorem \ref{Tnew}}]\hfill

\noindent {\bf Proof of (1).}
Suppose first that $\beta \in \Z$ and $\alpha \in |\beta| -2 +2\N$, say
$\alpha = |\beta| + 2m$ with $m \geq 0$. Note that $A = (a_{st}) \in
\bp_n(\C)$ if and only if $\overline{A} := (\overline{a_{st}}) \in
\bp_n(\C)$. Then,
\begin{equation}
\Psi_{\alpha, \beta}[A] = \begin{cases}
\Psi_{2m,0}[A] = (A \circ \overline{A})^{\circ m}, & \textrm{ if } \beta
= 0, \\
\Psi_{2m + \beta,\beta}[A] = A^{\circ \beta} \circ (A \circ
\overline{A})^{\circ m}, & \textrm{ if } \beta > 0, \\
\Psi_{2m + |\beta|,\beta}[A] = \overline{A}^{\circ |\beta|} \circ (A
\circ \overline{A})^{\circ m} , & \textrm{ if } \beta < 0.
\end{cases}
\end{equation}

\noindent In all three cases, we obtain that $\Psi_{\alpha,\beta}[A] \in
\bp_n(\C)$ by the Schur product theorem.

Suppose instead $\beta \in \Z$ and $\alpha \geq \max(n-2, |\beta|+2n-6)$.
We claim that in that case, $\Psi_{\alpha,\beta}[-]$ also preserves
Loewner positivity on $\bp_n(\C)$. The proof is by induction on $n \geq
3$. For $n=3$ we are done by Theorem \ref{Tzhan}. Now suppose the
assertion holds for $n-1 \geq 3$. Then $\Psi_{\alpha,1}[-]$ preserves
Loewner positivity on $\bp_{n-1}(\C)$ for $\alpha \geq 2(n-1-3) + 1 =
2n-7$. Hence by Proposition \ref{P01}, $\Psi_{\alpha,0}[-]$ preserves
Loewner positivity on $\bp_n(\C)$ for $\alpha \geq 2n-7 + 1 = 2n-6$. Thus
if $\alpha \geq 2n-6 + |\beta|$ and $A \in \bp_n(\C)$, then
\[ \Psi_{\alpha,|\beta|}[A] = \Psi_{\alpha-|\beta|,0}[A] \circ A^{\circ
|\beta|}, \qquad \Psi_{\alpha,-|\beta|}[A] = \Psi_{\alpha-|\beta|,0}[A]
\circ \overline{A}^{\circ |\beta|}, \]

\noindent and these are both in $\bp_n(\C)$ by the Schur product theorem.
Therefore the claim is proved by induction.

\noindent {\bf Proof of (2).}
If $\beta \not\in \Z$, then Theorem \ref{Tzhan} shows that $\Psi_{\alpha,
\beta}$ does not preserve Loewner positivity on $\bp_n(\C)$. Thus assume
$\beta \in \Z$. If $\alpha < 1$, it is easy to see that $\Psi_{\alpha,
\beta}[-]$ does not preserve positivity on $\bp_n(\C)$ (see Equation
\eqref{Etop3}). It thus remains to prove that $\Psi_{\alpha,\beta}[-]$
does not preserve positivity on $\bp_n(\C)$ if $1 \leq \alpha < |\beta| +
2 \lfloor (\sqrt{8n+1}-5)/2 \rfloor$, but $\alpha - |\beta|$ is not a
nonnegative even integer. To show this statement, first note for each
integer $k \geq 0$ that
\[ \lfloor (\sqrt{8n+1}-5)/2 \rfloor \geq k \qquad \iff \qquad n \geq
\binom{k+3}{2}. \]

\noindent Thus, we first show the assertion for $n = \binom{k+3}{2}$,
from which it immediately follows for all $n > \binom{k+3}{2}$ by padding
with zeros. Moreover, it suffices to show that $\Psi_{\alpha,\beta}[-]$
does not preserve Loewner positivity on $\bp_n(\C)$ when $\alpha \in
(|\beta| + 2k-2, |\beta| + 2k)$, since the smaller values of $\alpha \in
(|\beta|, |\beta| + 2k) \setminus (\alpha - 2\Z)$ do not preserve
positivity on $\bp_n(\C)$ by considering lower values of $k$ (and then
padding by zeros).

Thus, suppose $n = \binom{k+3}{2}$ and $\alpha \in (|\beta| + 2k-2,
|\beta| + 2k)$. It suffices to show that $\Psi_{\alpha,\beta}[-]$ does
not preserve positivity on $\bp_n(\C)$. Since $\Psi_{\alpha,-\beta}[A] =
\overline{\Psi_{\alpha,\beta}[A]}$, we may assume $\beta \geq 0$. Now fix
$z \in \C^\times$ and consider the function $f : (-1/|z|,1/|z|) \to \C$,
given by:
\[ f(\epsilon) := \Psi_{\alpha,\beta}(1 + \epsilon z) = (1 + \epsilon
z)^{(\alpha + \beta)/2} (1 + \epsilon \overline{z})^{(\alpha-\beta)/2}.
\]

\noindent Defining $Z(\epsilon) := 1 + \epsilon z$, one has:
\[ \frac{df}{d \epsilon} = \frac{d \Psi_{\alpha,\beta}(Z(\epsilon))}{d
\epsilon} = \frac{\partial \Psi_{\alpha,\beta}}{\partial Z} \frac{dZ}{d
\epsilon} + \frac{\partial \Psi_{\alpha,\beta}}{\partial \overline{Z}}
\frac{d\overline{Z}}{d \epsilon}. \]

\noindent Repeatedly using this formula and the general Leibniz rule, we
obtain for any integer $l \geq 0$:
\begin{align*}
\frac{d^l f}{d \epsilon^l}(0) = &\ \sum_{j=0}^l \binom{l}{j}
\prod_{t=0}^{j-1} \left( \frac{\alpha+\beta}{2} - t \right)
\prod_{t=0}^{l-j-1} \left( \frac{\alpha-\beta}{2} - t \right) \cdot
\left. \frac{z^j \overline{z}^{l-j} f(\epsilon)}{(1 + \epsilon z)^j (1 +
\epsilon \overline{z})^{l-j}} \right|_{\epsilon=0}\\
= &\ \sum_{j=0}^l \binom{l}{j} \Psi_{l,l-2j}(z) \prod_{t=0}^{j-1} \left(
\frac{\alpha+\beta}{2} - t \right) \prod_{t=0}^{l-j-1} \left(
\frac{\alpha-\beta}{2} - t \right).
\end{align*}

\noindent Therefore by Taylor's theorem, as $\epsilon \to 0^+$ we have
\begin{align}
\Psi_{\alpha,\beta}(1 + \epsilon z) = &\ 1 + \sum_{l=1}^{k+1}
\sum_{j=0}^l \frac{c_{l,j} \epsilon^l}{l!} \Psi_{l,l-2j}(z) +
o(\epsilon^{k+2}),\\
\mbox{where} \quad c_{l,j} := &\ \binom{l}{j} \prod_{t=0}^{j-1} \left(
\frac{\alpha+\beta}{2} - t \right) \prod_{t=0}^{l-j-1} \left(
\frac{\alpha-\beta}{2} - t \right)\ \forall 1 \leq l \leq k+1,\ 0 \leq j
\leq l.\notag
\end{align}

\noindent Now consider the family of power functions $S_k := \{
\Psi_{l,l-2j} : 1 \leq l \leq k+1, 0 \leq j \leq l \} \cup \{ K \equiv
1\}$. Note that $S_k$ contains precisely $\binom{k+3}{2}$ functions,
which are linearly independent on $\C^n$ by Lemma \ref{LCmult}. Hence
there exists a vector $u_{k,n} \in \C^n$ such that
\begin{equation}\label{Echoiceofu}
\Psi_{k+1,k+1}[u_{k,n}] \notin {\rm span}_{\C} \{ h[u_{k,n}] : h \in S_k
\setminus \{ \Psi_{k+1,k+1} \} \}.
\end{equation}

\noindent Now define the matrix $A_\epsilon := {\bf 1}_{n \times n} +
\epsilon u_{k,n} u_{k,n}^* \in \bp_n(\C)$. Then,
\[ \Psi_{\alpha,\beta}[A_\epsilon] = {\bf 1}_{n \times n} +
\sum_{l=1}^{k+1} \frac{c_{l,j} \epsilon^l}{l!} \Psi_{l,l-2j}[u_{k,n}]
\Psi_{l,l-2j}[u_{k,n}]^* + o(\epsilon^{k+2}) C, \]

\noindent where $C_{n \times n}$ is a fixed matrix independent of
$\epsilon$. Moreover, there exists $v_{k,n} \in \C^n$ orthogonal to $\{
h[u_{k,n}] : h \in S_k \setminus \{ \Psi_{k+1,k+1} \} \}$, but not to
$\Psi_{k+1,k+1}[u_{k,n}]$. Now compute:
\begin{align*}
v_{k,n}^* \Psi_{\alpha,\beta}[A_\epsilon] v_{k,n} = &\ \frac{c_{k+1,0}
\epsilon^{k+1}}{(k+1)!} |v_{k,n}^* \Psi_{k+1,k+1}[u_{k,n}]|^2 +
o(\epsilon^{k+2}) v_{k,n}^* C v_{k,n}\\
= &\ \frac{|v_{k,n}^* \Psi_{k+1,k+1}[u_{k,n}]|^2}{2^{k+1}(k+1)!} \cdot
\epsilon^{k+1} \prod_{t=0}^k (\alpha - \beta - 2t) + o(\epsilon^{k+2})
v_{k,n}^* C v_{k,n}.
\end{align*}

\noindent Since $\alpha \in (\beta + 2k-2, \beta+2k)$, the first term
is negative, whence so is the entire expression for sufficiently small
$\epsilon > 0$. This shows that $\Psi_{\alpha,\beta}[-]$ does not
preserve Loewner positivity on $\bp_n(\C)$ if $\alpha \in (\beta+2k-2,
\beta+2k)$, which concludes the proof.
\end{proof}

\begin{remark}
Since $n \geq \binom{k+3}{2}$, we observe that the vector $u_{k,n} \in
\C^n$ satisfying \eqref{Echoiceofu} can in fact be chosen to have all its
entries in the complex disc $D(0,R)$ for any fixed $0 < R \leq \infty$.
Indeed, by Lemma \ref{LCmult}, the characters in the set $S_k$ are
linearly independent on $D(0,R)$. Thus there exists $u = u_{k,n} \in
D(0,R)^n$ such that the vectors $\{ h[u] : h \in S_k \}$ are linearly
independent.
\end{remark}

\bibliographystyle{plain}
\bibliography{biblio}

\end{document}